\documentclass[11pt,reqno]{amsart}

\usepackage[margin=1in]{geometry}
\usepackage[tt=false]{libertine}

\usepackage{amsmath,amsfonts,amssymb,amsthm}

\usepackage{mathabx}

\usepackage{mathrsfs}

\usepackage{graphicx}
\usepackage{enumerate}
\usepackage{bm}

\usepackage{url}
\usepackage{bbm}

\usepackage{verbatim}
\usepackage{hyperref,color}
\usepackage[capitalize,nameinlink]{cleveref}
\usepackage[dvipsnames]{xcolor}
\hypersetup{
	colorlinks=true,
	pdfpagemode=UseNone,
    citecolor=OliveGreen,
    linkcolor=NavyBlue,
    urlcolor=black,
	pdfstartview=FitW
}
\usepackage{appendix}
\crefname{appsec}{Appendix}{Appendices}
\usepackage{tikz}

\theoremstyle{plain}
\newtheorem{theorem}{Theorem}

\newtheorem{lemma}[theorem]{Lemma}

\newtheorem*{conjecture*}{Conjecture}

\theoremstyle{definition}

\newtheorem*{assumption*}{Assumption}

\theoremstyle{remark}

\crefname{lemma}{Lemma}{Lemmas}
\crefname{theorem}{Theorem}{Theorems}
\crefname{definition}{Definition}{Definitions}
\crefname{fact}{Fact}{Facts}
\crefname{claim}{Claim}{Claims}
\crefname{proposition}{Proposition}{Propositions}

\newcommand{\E}{\mathbb{E}}

\renewcommand{\epsilon}{\varepsilon}

\newcommand{\PP}{\mathbb{P}}

\newcommand{\beq}{\begin{equation}}
\newcommand{\eeq}{\end{equation}}

\begin{document}
	
\title{On the Second Kahn--Kalai Conjecture}

\author[E.\ Mossel, J.\ Niles-Weed, N.\ Sun, I.\ Zadik]{Elchanan Mossel$^{\star\circ}$, Jonathan Niles-Weed$^\dagger$, Nike Sun$^\star$, and Ilias Zadik$^\star$}
\thanks{$^\star$Department of Mathematics, MIT;
$^\circ$MIT Institute for Data, Systems, and Society;
$^\dagger$Center for Data Science \& Courant Institute of Mathematical Sciences, NYU. Email: \texttt{\{elmos,nsun,izadik\}@mit.edu}; \texttt{jnw@cims.nyu.edu}}

\date{ \vspace{0.2cm}
\today}

\newcommand{\pcrit}{p_\mathsf{c}}
\newcommand{\qfrac}{q_\mathsf{f}}
\newcommand{\pE}{p_\mathsf{E}}
\newcommand{\pEnew}{\tilde{p}_\mathsf{E}}

\begin{abstract} 
For any given graph $H$, we are interested in $\pcrit(H)$, the minimal $p$ such that the Erd\H{o}s--R\'enyi graph $G(n,p)$ contains a copy of $H$ with probability at least $1/2$. Kahn and Kalai (2007) conjectured that $\pcrit(H)$ is given up to a logarithmic factor by a simpler ``subgraph expectation threshold'' $\pE(H)$, which is the minimal $p$ such that for every subgraph $H'\subseteq H$, the Erd\H{o}s--R\'enyi graph $G(n,p)$ contains \emph{in expectation} at least $1/2$ copies of $H'$. It is trivial that $\pE(H)  \le  \pcrit(H)$, and the so-called ``second Kahn--Kalai conjecture'' states that $\pcrit(H) \lesssim \pE(H) \log e(H)$ where $e(H)$ is the number of edges in $H$.

In this article we present a natural modification $\pEnew(H)$ of the Kahn--Kalai subgraph expectation threshold, which we show is sandwiched between $\pE(H)$ and $\pcrit(H)$. The new definition $\pEnew(H)$ is based on the simple observation that if $G(n,p)$ contains a copy of $H$ and $H$ contains \emph{many} copies of $H'$, then $G(n,p)$ must also contain \emph{many} copies of $H'$. We then show that $\pcrit(H) \lesssim \pEnew(H) \log e(H)$, thus proving a modification of the second Kahn--Kalai conjecture. The bound follows by a direct application of the set-theoretic``spread''  property,  which led to recent breakthroughs in the sunflower conjecture by Alweiss, Lovett, Wu and Zhang and the first fractional Kahn--Kalai conjecture by Frankston, Kahn, Narayanan and Park.

\end{abstract}

\maketitle

\section{Introduction}

In this work, we are interested in the following fundamental question. 
Given a graph $H$, what is the \textit{smallest} value $p=\pcrit(H)$ for which the Erd\H{o}s--R\'enyi graph $G=G(n,p)$ contains an isomorphic copy of $H$ with probability at least $1/2$?\footnote{The graph $H$ is allowed to depend on $n$. Indeed, when $e(H)$ does not grow with $n$ 
the value of $\pcrit(H)$ is known \cite{MR125031, bela_fixed}, so the main contribution of the present paper is in the setting where $e(H)$ grows with $n$.} The value 
$\pcrit(H)$ is often referred to as the ``critical threshold'' for the appearance of $H$. 
A well-known conjecture from \cite{kahn2007thresholds} posits that $\pcrit(H)$ 
is given up to a logarithmic factor by a simpler \emph{``subgraph expectation threshold''} $\pE(H)$, which is the \textit{maximum first-moment threshold among all subgraphs of $H$}.\footnote{Throughout, a ``subgraph'' of a graph $H$ refers to the \textit{edge-induced} subgraph associated with a subset of the edges of $H$.} More precisely, for any (labelled) graphs $H$ and $H’$, let $M_{H’,H}$ denote the number of subgraphs of $H$ which are isomorphic copies of $H’$, and define
	\beq\label{e:pE}
	\pE(H)
	=\min\bigg\{ p : \E M_{H',G(n,p)} \ge \frac12
	\textup{ for all } H' \subseteq H\bigg\}\,.
	\eeq
It is a trivial consequence of Markov's inequality {(see \S\ref{ss:basicobs})} that
$\pE(H)$ lower bounds $\pcrit(H)$. Kahn and Kalai proposed that this easy lower bound may not be far from the truth:

\begin{conjecture*}[{\cite[Conjecture~2]{kahn2007thresholds}}]
It holds that $\pcrit(H) \le L \pE(H) \log e(H)$ for a universal constant $L$.
\end{conjecture*}

The factor $\log e(H)$ is necessary in some important examples, such as when $H$ is a perfect matching or Hamiltonian cycle. In both cases $\pE(H) \asymp 1/n$ but $\pcrit(H)\asymp (\log n)/n$ (\cite{MR125031,erdHos1966existence, MR389666,MR0434878}); for more details see the discussion in \cite{kahn2007thresholds}. The above conjecture remains open, although related conjectures of \cite{kahn2007thresholds,MR2743011} were proved recently (\cite{fracKK_annals,park2022proof}), {as we review below.}

\subsection{Main result}

In this article we introduce a natural variant $\pEnew(H)$ of $\pE(H)$, and show that it captures $\pcrit(H)$ up to a logarithmic factor, thus proving a modification of \cite[Conjecture~2]{kahn2007thresholds}. The modification is based on the simple observation that if $G=G(n,p)$ contains a copy of $H$, then we must have $M_{H',G} \ge M_{H',H}$ for any subgraph $H'$ of $H$ --- in contrast with the weaker bound $M_{H',G}\ge1$, which is used to show $\pE(H) \le \pcrit(H)$. Thus, if we define the  \emph{``modified subgraph expectation threshold''} 
	\beq\label{e:pEnew}
	\pEnew(H)
	=\min\bigg\{ p : \E M_{H',G(n,p)} \ge \frac{M_{H',H}}{2}
	\textup{ for all } H' \subseteq H\bigg\}\,,
	\eeq
then it is easy to see that $\pE(H) \le \pEnew(H) \le \pcrit(H)$ {(see also \S\ref{ss:basicobs} below).} Our main result is that the
new lower bound $\pEnew(H)\le\pcrit(H)$ is tight up to a logarithmic factor:

\begin{theorem}\label{thm:main}
It holds that $\pcrit(H) \le L \pEnew(H)\log e(H)$ for a universal constant $L$.
\end{theorem} 

A straightforward but tedious calculation gives that $\pEnew(H)\asymp 1/n$ when $H$ is a Hamiltonian cycle. Therefore, as with
\cite[Conjecture~2]{kahn2007thresholds}, the factor $\log e(H)$ is indeed necessary for Theorem~\ref{thm:main} to hold.

\subsection{Comparison with previous work}

The works most closely related to ours arise from the study of  \cite[Conjecture~1]{kahn2007thresholds}. 
 This ``first Kahn--Kalai conjecture'' applies more broadly to the setting of \emph{all monotone properties}, but is weaker than the second Kahn--Kalai conjecture (\cite[Conjecture~2]{kahn2007thresholds}) in the setting of \emph{graph inclusion properties} (which are the focus of this article). The first Kahn--Kalai conjecture states that for any monotone property $\mathscr{F}\subseteq\{0,1\}^X$, we have
	\[
	\pcrit(\mathscr{F}) \le q(\mathscr{F}) \log \ell(\mathscr{F})
	\]
where $q(\mathscr{F})$ is the \emph{maximum first moment threshold among all covers of $\mathscr{F}$}, and $\ell(\mathscr{F})$ is the size of a largest minimal element of $\mathscr{F}$. Talagrand \cite{MR2743011} proposed a relaxation of the above, the so-called ``fractional Kahn--Kalai conjecture''
	\[
	\pcrit(\mathscr{F}) 
	\le \qfrac(\mathscr{F}) \log \ell(\mathscr{F})
	\]
where $q(\mathscr{F})$ is the \emph{maximum first moment threshold among all fractional covers of $\mathscr{F}$}. It is trivial that $q(\mathscr{F})\le \qfrac(\mathscr{F})\le \pcrit(\mathscr{F})$. Both conjectures were long-standing open problems, which were resolved only recently in two notable works \cite{fracKK_annals,park2022proof}.

 In comparison, the second Kahn--Kalai conjecture is an even stronger conjecture in the particular setting of graph inclusion properties; Kahn and Kalai referred to it as their ``starting point'' in formulating their first conjecture. If $\mathscr{F}_H$ is the property of containing a copy of some graph $H$, then the threshold $\pE(H)$ of \eqref{e:pE} is the \emph{maximum first moment threshold among all ``subgraph containment covers'' of $\mathscr{F}_H$}. Therefore $\pE(H) \le q(\mathscr{F})$, and in the graph inclusion setting the first Kahn--Kalai conjecture may be viewed as a relaxation of the second. To the best of our knowledge, beyond the results on the first conjecture, no further progress has been made on the second one; and it has been reiterated in various places \cite{filmus2014real,fracKK_annals}. In this work, we modify and prove the second Kahn--Kalai conjecture (Theorem \ref{thm:main}). It is an interesting question whether Theorem~\ref{thm:main} can be used to prove (or disprove) the second Kahn--Kalai conjecture.

Interestingly, \emph{for graph inclusion properties}, our result slightly improves on the fractional Kahn--Kalai conjecture (posed by \cite{MR2743011} and proved by \cite{fracKK_annals} for general monotone properties). Indeed, our modified subgraph expectation threshold $\pEnew(H)$ can be interpreted as the 
\emph{maximum first moment threshold among certain
``subgraph containment fractional covers'' of $\mathscr{F}_H$}. To make the correspondence, using the notation of \cite{fracKK_annals}, for any subgraph $H'$ of $H$ one can assign weight $g_{H'}(S)=1/M_{H',H}$ to any subgraph $S\subseteq K_n$ that is a copy of $H'$. This leads to a fractional cover of $\mathscr{F}_H$ whose first-moment threshold is the unique $p$ satisfying $\E M_{H',G(n,p)} = M_{H',H}/2$.  It follows that $\pEnew(H) \le \qfrac(\mathscr{F}_H)$. Thus, our result implies the fractional Kahn--Kalai bound for graph inclusion properties, in fact with an explicit choice of fractional covers. Whether our result implies the original first Kahn--Kalai conjecture remains an interesting open problem.

\section{Proofs}

\subsection{Basic notations and calculations}
\label{ss:basicobs}
 Recall that $M_{H',H}$ denotes the number of (labelled) subgraphs of $H$ which are copies of $H'$. We abbreviate $M_H\equiv M_{H,K_n}$ where $K_n$ is the complete graph on $n$ vertices. We also abbreviate $Z_H\equiv M_{H,G}$ where $G$ is the Erd\H{o}s--R\'enyi graph $G(n,p)$. Writing $\PP_p$ for the law of $G(n,p)$, recall that
	\[
	\pcrit(H)
	= \inf\bigg\{ p : \PP_p(Z_H\ge1) \ge \frac12\bigg\}\,.
	\]
If $p \ge \pcrit(H)$, then Markov's inequality gives
	\[
	\frac12 \le \PP_p(Z_H\ge1) \le \PP_p\Big(Z_{H'}\ge1
		\ \forall H' \subseteq H\Big)
	\le \min\bigg\{ \E Z_{H'} : H'\subseteq H\bigg\}
	\,,
	\]
which implies $\pE(H) \le \pcrit(H)$ for $\pE(H)$ as defined by \eqref{e:pE}. Our definition \eqref{e:pEnew} of $\pEnew(H)$ is based on the simple observation that in fact $p \ge \pcrit(H)$ together with Markov's inequality implies
	\[
	\frac12 \le \PP_p(Z_H\ge1) \le \PP_p\Big(Z_{H'}\ge M_{H',H}
		\ \forall H' \subseteq H\Big)
	\le \min\bigg\{ \frac{\E Z_{H'}}{M_{H',H}}
	 : H'\subseteq H\bigg\}
	\,,
	\]
therefore $\pEnew(H) \le \pcrit(H)$. 
It is clear that
$\pE(H) \le\pEnew(H)$; moreover, if $\E_p$ denotes expectation with respect to $\PP_p$, then
 $\E_p Z_{H'} = M_{H'} p^{e(H')}$, so we can rewrite
 \eqref{e:pE} as
 	\[
	\pE(H)
	= \max\bigg\{
	\bigg(\frac{1}{2M_{H'}}\bigg)^{1/e(H')}
	: H'\subseteq H\bigg\}\,,
	\]
and likewise we can rewrite \eqref{e:pEnew} as
	\beq\label{e:gen.first.mmt.threshold}
	\pEnew(H)
	= \max\bigg\{
	\bigg(\frac{M_{H',H}}{2M_{H'}}\bigg)^{1/e(H')}
	: H'\subseteq H\bigg\}\,.
	\eeq

\subsection{The spread lemma} 
The proof of Theorem \ref{thm:main} is an application of a powerful tool, which we refer to as the ``spread lemma.'' Various forms of this lemma have played a key role in establishing recent breakthrough results on the sunflower theorem \cite{sunflower_annals} (see also \cite{Rao_sunflower, Tao_sunflower}) and the proof of the {fractional Kahn--Kalai conjecture}  \cite{fracKK_annals}.

To state the lemma in our setting, let $\pi$ be an arbitrary distribution over subgraphs of $K_n$ (e.g., the copies of particular subgraph of $K_n$). For $R>1$, we say that $\pi$ is \textit{$R$-spread} if for all (without loss of generality, nonempty) subgraphs $J_0\subseteq K_n$,  if $\bm{H}\sim\pi$ then
\beq\label{spread_cond}
\pi(J_0 \subseteq \bm{H}) \leq R^{-e(J_0)}.
\eeq
 Then the spread lemma as stated in \cite[Theorem 1.6]{fracKK_annals} applied to graph inclusion properties implies the following result.

\begin{lemma}\label{spread_lemma}
Fix integers $k,M \geq 1$. Let $G_1,\ldots,G_{M}$ be subgraphs of $K_n$ with $e(G_i) \leq k, i \in [M]$. For some universal constant $C>0$, if the uniform distribution $\pi$ over $\{G_1,\ldots,G_M\}$ is $R$-spread and  $p>C\frac{ \log k}{R} $, then a sample from $G(n,p)$ contains one of the $G_i$'s with probability at least $1/2$.
\end{lemma}

\begin{proof}We choose a sufficiently large constant $C>K$ where $K$ is the universal constant from \cite[Theorem 1.6]{fracKK_annals}. 
From standard concentration results, when $C>K$ is large enough, a sample from $G(n,p)$ contains, with probability at least $2/3$,
 a uniformly random undirected graph on $n$ vertices and $K\frac{ \log k}{R}\binom{n}{2}$ edges. The result then follows directly from \cite[Theorem 1.6]{fracKK_annals} applied to $X=K_n$ and $\kappa=R$.
\end{proof}

\subsection{The proof}

We now show that Theorem~\ref{thm:main} follows easily from Lemma~\ref{spread_lemma}:

\begin{proof}[Proof of Theorem~\ref{thm:main}]
Let $\pi\equiv\pi_H$ denote the uniform distribution over all the copies of $H$ in the complete graph $K_n$, and let $\bm{H}$ denote a sample from $\pi$. For a {nonempty} $J \subseteq H$, let $\pi_J$ denote the uniform distribution over all the copies of $J$ in $K_n$, and let $\bm{J}$ denote a sample from $\pi_J$. Let $J_0,H_0$ be any fixed copies of $J,H$ respectively in $K_n$. Then we have
	\beq\label{double_count} 
	\pi_H(J_0\subseteq\bm{H})
	=\pi_J( \bm{J} \subseteq H_0 )
	= \frac{M_{J,H}}{M_J}\,.
	\eeq
(The first equality is by symmetry. The second holds because there are $M_J$ possibilities for $\bm{J}$, of which $M_{J,H}$ are contained in $H_0$.) By combining \eqref{double_count}  with the definition \eqref{e:gen.first.mmt.threshold} of $\pEnew(H)$, we find
	\[
	\pi_H(J_0\subseteq\bm{H})
	\stackrel{\eqref{double_count}}{=}
		\frac{M_{J,H}}{M_J}
	\stackrel{\eqref{e:gen.first.mmt.threshold}}{\le}
		2 \pEnew(H)^{e(J)}
	\le \bigg(
	\frac1{2 \pEnew(H)}
	\bigg)^{-e(J)}\,,
	\]
where the last inequality uses that $e(J) \ge1$. Since $J_0$ is arbitrary, we conclude that $\pi$ is $R$-spread with
$R= 1/(2\pEnew(H))$. An appeal to Lemma \ref{spread_lemma} with $k=e(G)$ concludes the proof.
\end{proof}

\section*{Acknowledgements}
We thank Youngtak Sohn for helpful feedback on an earlier draft of this work. I.Z. thanks Dan Mikulincer for helpful discussions. 
We also acknowledge the support of
Simons-NSF grant DMS-2031883 (E.M., N.S., and I.Z.),
the Vannevar Bush Faculty Fellowship ONR-N00014-20-1-2826 (E.M.\ and I.Z.), the Simons Investigator Award 622132 (E.M.),
the Sloan Research Fellowship (J.N.W.),
 and NSF CAREER grant DMS-1940092 (N.S.).

\vfill

{\raggedright
\bibliographystyle{alphaabbr}
\bibliography{pc}}

\newcommand{\etalchar}[1]{$^{#1}$}
\begin{thebibliography}{ALWZ21}

\bibitem[ALWZ21]{sunflower_annals}
R.~Alweiss, S.~Lovett, K.~Wu, and J.~Zhang.
\newblock Improved bounds for the sunflower lemma.
\newblock {\em Ann. of Math. (2)}, 194(3):795--815, 2021.

\bibitem[Bol81]{bela_fixed}
B.~Bollob{\'a}s.
\newblock Random graphs.
\newblock {\em London Math. Soc. Lec. Note Series}, 52:80--102, 1981.

\bibitem[ER60]{MR125031}
P.~Erd\H{o}s and A.~R\'{e}nyi.
\newblock On the evolution of random graphs.
\newblock {\em Magyar Tud. Akad. Mat. Kutat\'{o} Int. K\"{o}zl.}, 5:17--61,
  1960.

\bibitem[ER66]{erdHos1966existence}
P.~Erd{\H{o}}s and A.~R{\'e}nyi.
\newblock On the existence of a factor of degree one of a connected random
  graph.
\newblock {\em Acta Math. Acad. Sci. Hungar}, 17:359--368, 1966.

\bibitem[FHH{\etalchar{+}}14]{filmus2014real}
Y.~Filmus, H.~Hatami, S.~Heilman, E.~Mossel, R.~O’Donnell, S.~Sachdeva,
  A.~Wan, and K.~Wimmer.
\newblock Real analysis in computer science: A collection of open problems.
\newblock {\em Preprint available at https://simons. berkeley.
  edu/sites/default/files/openprobsmerged. pdf}, 2014.

\bibitem[FKNP21]{fracKK_annals}
K.~Frankston, J.~Kahn, B.~Narayanan, and J.~Park.
\newblock Thresholds versus fractional expectation-thresholds.
\newblock {\em Ann. of Math. (2)}, 194(2):475--495, 2021.

\bibitem[KK07]{kahn2007thresholds}
J.~Kahn and G.~Kalai.
\newblock Thresholds and expectation thresholds.
\newblock {\em Combin. Probab. Comput.}, 16(3):495--502, 2007.

\bibitem[Kor76]{MR0434878}
A.~D. Kor\v{s}unov.
\newblock Solution of a problem of {P}. {E}rd{\H{o}}s and {A}. {R}\'{e}nyi on
  {H}amiltonian cycles in undirected graphs.
\newblock {\em Dokl. Akad. Nauk SSSR}, 228(3):529--532, 1976.

\bibitem[P{\'o}s76]{MR389666}
L.~P{\'o}sa.
\newblock Hamiltonian circuits in random graphs.
\newblock {\em Discrete Math.}, 14(4):359--364, 1976.

\bibitem[PP22]{park2022proof}
J.~Park and H.~T. Pham.
\newblock A proof of the {K}ahn--{K}alai conjecture.
\newblock {\em arXiv preprint arXiv:2203.17207}, 2022.

\bibitem[Rao20]{Rao_sunflower}
A.~Rao.
\newblock Coding for sunflowers.
\newblock {\em Discrete Analysis}, (2), 2020.

\bibitem[Tal10]{MR2743011}
M.~Talagrand.
\newblock Are many small sets explicitly small?
\newblock In {\em Proc.\ 42nd STOC}, pages 13--35. ACM, New York, 2010.

\bibitem[Tao20]{Tao_sunflower}
T.~Tao.
\newblock The sunflower lemma via {S}hannon entropy.
\newblock Blog post,
  \url{https://terrytao.wordpress.com/2020/07/20/the-sunflower-lemma-via-shannon-entropy},
  2020.

\end{thebibliography}
\end{document}